\documentclass[twoside,11pt]{amsart}

\usepackage{amsmath,latexsym,amssymb,mathptm, verbatim, enumerate, ytableau}
    
       \usepackage{amsfonts,epsfig,textcomp,mathdots,times}%,mathptm}

\usepackage{mathtools}

\usepackage{eucal}
\usepackage{tikz}
\usepackage{tikz-cd}
\usetikzlibrary{%
  matrix,%
  calc,%
  arrows%
}
\usepackage{dsfont}
\usetikzlibrary{matrix}

\usepackage[all,cmtip]{xy}

\input amssym.def
\input amssym
\input xypic
\input xy

\xyoption{all}
\def\demo{\noindent{\bf Proof. }}
\def\sqr#1#2{{\vcenter{\hrule height.#2pt
        \hbox{\vrule width.#2pt height#1pt \kern#1pt
                \vrule width.#2pt}
        \hrule height.#2pt}}}
\def\square{\mathchoice\sqr64\sqr64\sqr{4}3\sqr{3}3}
\def\QED{\hfill$\square$}

\def\ann{\operatorname{ann}}

\def\lto{\longrightarrow}

\define\Hom{\operatorname{Hom}}

\define\ker{\operatorname{ker}}

\define\im{\operatorname{im}}

\def\Ext{\operatorname{Ext}}

\def\C{\mathcal C}

\def\a{\mathfrak a}

\def\m{\mathfrak m}

\def\htt{\operatorname{ht}}

\def\Supp{\operatorname{Supp}}
\def\ann{\operatorname{ann}}

\def\bs{\bigskip}

\def\m{\mathfrak m}
\def\a{\mathfrak a}

\def\dim{{\rm dim \, }}

\def\bs{\bigskip}
\def\ms{\medskip}

\def\lto{\longrightarrow}

\newtheorem{theorem}{Theorem}[section]
\newtheorem{lemma}[theorem]{Lemma}
\newtheorem{corollary}[theorem]{Corollary}
\newtheorem{proposition}[theorem]{Proposition}

\newtheorem{question}[theorem]{Question}

\newtheorem{def&dis}[theorem]{Definition and Discussion}
\theoremstyle{definition}
\newtheorem{definition}[theorem]{Definition}
\newtheorem{examples}[theorem]{Examples}
\newtheorem{remark}[theorem]{Remark}

\newtheorem{example}[theorem]{Example}

\newtheorem{setting}[theorem]{Setting}

\def\noi{\noindent}
\begin{document}

\baselineskip=16pt

\title[Quasi-cyclic modules and coregular sequences]{ 
Quasi-cyclic modules and coregular sequences}
\date\today

\author[R. Hartshorne \and C. Polini]
{Robin Hartshorne  \and Claudia Polini}

\thanks{AMS 2010 {\em Mathematics Subject Classification}.
Primary 13D45,14B15; Secondary 14M10.}


\thanks{The second author was partially supported by NSF grant DMS-1601865.}

\address{Department of Mathematics, University of California, Berkeley, California, 94720-3840}\email{robin@math.berkeley.edu}
\address{Department of Mathematics, University of Notre Dame, Notre Dame, Indiana 46556} \email{cpolini@nd.edu}

\begin{abstract}  We develop the theory of coregular sequences and codepth for modules that need not be finitely generated or artinian over a Noetherian ring. We use this theory to give a new version of a theorem of Hellus characterizing set-theoretic complete intersections in terms of local cohomology modules. We also define quasi-cyclic modules as increasing unions of cyclic modules, and show that modules of codepth at least two  are quasi-cyclic. We then focus our attention on curves in $\mathbb P^3$ and give  a number of necessary conditions for a  curve to be a set-theoretic complete intersection. Thus an example of a curve for which any of these necessary conditions does not hold would provide a negative answer to the still open problem, whether every connected curve in $\mathbb P^3$ is a set-theoretic complete intersection. 
\end{abstract}

\maketitle

%\tableofcontents

\section{Introduction}

\medskip

Kronecker \cite{Kr} in 1882 proved a theorem, which in modern language says that any closed algebraic subset of $\Bbb P^n$ can be cut out (set-theoretically) by $n+1$ hypersurfaces. An easy proof was given later by van der Waerden, choosing successively hypersurfaces so that the excess intersection drops in dimension each time until  it is empty.

Vahlen \cite{V} in 1891 published an example to show that Kronecker's bound was best possible. It is a rational quintic curve with a single quadrisecant in $\Bbb P^3$, which he said could not be the intersection of three surfaces.

Thus the problem of finding the minimal number of surfaces needed to cut out a curve in $\Bbb P^3$ was solved, for fifty years, until Perron \cite{P} in 1941 showed that Vahlen's example was wrong. Indeed, Perron exhibited three quintic surfaces whose intersection was Vahlen's curve. More generally, Kneser \cite{Kn} showed in 1960 that any curve in $\Bbb P^3$ could be cut out by three surfaces, and later Eisenbud and Evans \cite{EE} proved an algebraic result, generalizing Kneser's method, which shows that any algebraic subset of $\Bbb P^n$ or $\Bbb A^n$ can be cut out by $n$ hypersurfaces. 

Still, the question whether every irreducible curve in $\Bbb P^3$ is the intersection of two surfaces, in which case we say it is a {\it set-theoretic complete intersection}, remains open. 

Hellus \cite{He} in 2006 gave a criterion for a variety $V$ in $\Bbb P^n$ to be a set-theoretic complete intersection, in terms of the depth of the Matlis dual $D(H^r_I(A))$ of its only non-vanishing local cohomology module. Here $A$ is the coordinate ring of $\Bbb P^n$ and $I$ is the ideal of the variety $V$ of codimension $r$. Hellus's criterion has the advantage of showing that the question whether $V$ is a set-theoretic complete intersection depends only on module-theoretic properties of the module $M=H^r_I(A)$. However, it is impractical, since to find polynomials $f_1, \ldots, f_r\in A$ that form a regular sequence for $D(M)$ is tantamount to finding $f_1, \ldots, f_r$ so that $\sqrt{I}=\sqrt{(f_1, \ldots, f_r)}$, which is  simply a restatement of the problem.  

In this paper, in order to avoid dealing with Matlis duals of large modules, we define in Section~\ref{S1} the notion of {\it coregular sequences} and the corresponding notion of {\it codepth} for an $A$-module $M$. These notions, which have appeared earlier in the context of artinian $A$-modules only (see Example~\ref{ex2.5}), are in some sense dual to the usual notions of regular sequences and depth. We investigate the properties of coregular sequences using Koszul cohomology. 

We then give a new version of Hellus's theorem in  Section~\ref{S2}. It shows that a variety $V$ of codimension $r$ in $\Bbb P^n$ is a set-theoretic complete intersection if and only if $H^i_I(A)=0$ for all $i>r$ and $M=H^r_I(A)$ has codepth $r$. 

In Section~\ref{S3}, we  define the notion of a {\it quasi-cyclic module} to be an increasing union of cyclic submodules, or equivalently, a module in which any two elements are contained in some cyclic submodule. We show that a module of codepth $\ge 2$ is quasi-cyclic. Thus a necessary condition for a curve $\C$ in $\Bbb P^3$  to be a set-theoretic complete intersection is that the associated module $M=H^2_I(A)$ is quasi-cyclic. This property may be more amenable to verification than asking for its codepth. 

In Section~\ref{S4}, we relate the behavior of an element $h\in A$ acting on the local cohomology module $M=H^2_I(A)$ of a curve $\C$ in $\Bbb P^3$ to properties of the surface $X$ defined by $h=0$, and its open subset $X\setminus \C$. 

In Section~\ref{S5}, we interpret the property of a curve being defined set-theoretically by three equations in terms of the Koszul cohomology of $M$.

To sum up, in Section~\ref{S6}, we list a number of questions, starting with the motivating question of this paper: {\it Is every irreducible curve in $\Bbb P^3$ a set-theoretic complete intersection?} and then listing a number of consequences, so that a NO answer to any of these questions would give an example of a curve that is not a set-theoretic complete intersection. 
\bigskip

\section{Coregular sequences and codepth}\label{S1}

Throughout this paper, $A$ will denote a Noetherian local ring or a standard graded ring. If $A$ is a standard graded ring, we assume that all modules are graded and all ideals and elements are homogenous. 

\begin{definition} Let $A$ be a ring, $I$ an ideal, and $M$ a non-zero $A$-module with $\Supp(M)\subset V(I)$. 
A sequence  $x_1, \ldots, x_r$ of elements of $I$ is {\it coregular} for $M$ if  
\begin{enumerate}
\item $x_1M=M$
\item $x_{i+1} \, (0:_M(x_1, \ldots, x_{i}) )=0:_M(x_1, \ldots, x_{i})  \mbox{\quad  for\ } 1 \le i \le r-1\, .$
\end{enumerate}
\noi
In other words multiplication by $x_1$ is surjective on $M$, multiplication by $x_2$ is surjective on the kernel of the multiplication by $x_1$, and so on. 

\end{definition}

\begin{lemma}\label{ann} Let $A$ be a ring, $I$ an ideal, and $M$ a non-zero $A$-module with  $\Supp(M)\subset V(I)$.   If $x\in I$ and $m\in M$, then there exists a power of $x$ that kills $m$. 
\end{lemma}
\begin{proof}  Since $mA \subset M$, the support of $mA$ is contained in the support of $M$. Thus
$$V({\rm ann}_A m)=\Supp(mA)\subset \Supp(M)\subset V(I).$$ Hence $I\subset \sqrt{{\rm ann}_A m}$. 
\end{proof}

\begin{proposition}\label{finite} Let $A$ be a ring, $I$ an ideal, and $M$ a non-zero $A$-module with  $\Supp(M)\subset V(I)$. 
Then the length of any coregular sequence $x_1, \ldots, x_r$  in $I$ for $M$ is at most equal to the dimension of $A$. 
\end{proposition}
\demo We proceed by induction on $r\ge 1$. If $x\in I$ is a coregular element for $M$, then we will show that the dimension of $A$ is positive. We can mod out by the annihilator of $M$.  Write $\a=\ann M$ and $A'=A/\a$. We observe that the support of $M$ is still contained in $V(IA')$. Indeed, this is equivalent to $\Supp(M)\subset V(I+\a)$ and since $\a=\ann M$ we clearly have $\Supp(M)\subset V(\a)$. Thus $\Supp(M)\subset V(I) \cap V(\a)=V(I+\a)$.

We have now reduced to the case when $M$ is a faithful $A$-module. We first show that $x$ is a non zero-divisor on $A$. Suppose that $xy=0$ for some $y\in A$. We prove that $y=0$ by showing that $y$ is in the annihilator of $M$. Indeed, for every  $m\in M$ there exists an $m'\in M$ such that $m=xm'$ because $x$ is coregular for $M$. Thus $ym=yxm'=0$. 
Thus $\htt x A\ge 1$ and therefore the dimension of $A$ is at least one. 

For the general case of a coregular sequence, we first reduce to the case $M$ faithful so that $\htt x_1A\ge 1$. Then the ring $A'=A/xA$ has dimension at most $\dim A -1$. The elements $x_2, \ldots, x_r\in IA'$ are still coregular for $M'=0:_M x$. Notice the module $M'$ is non-zero by Lemma~\ref{ann} since the support of $M$ is contained in $V(I)$. Also $\Supp(M')\subset \Supp(M)\subset V(I)=V(IA')$, hence the support of $M'$ is still in $V(IA')$. By induction $\dim A'\ge r-1$, thus $\dim A\ge r$. 
\QED

\begin{definition} Let $A$ be a ring, $I$ an ideal, and $M$ a non-zero $A$-module with $\Supp(M)\subset V(I)$. The $I$-{\it codepth} of $M$ is the supremum of lengths of coregular sequences in $I$ for $M$. 
\end{definition}

Notice that by Proposition~\ref{finite}, if the ring $A$ has finite dimension then the  $I$-{\it codepth} of $M$ is finite. 

\begin{example}\label{ex2.5} Let $(A, \m)$ be a complete Noetherian local ring and let $M$ be an artinian non-zero $A$-module. Then the $\m$-codepth of $M$ is equal to the usual depth of the Matlis dual $D(M)$ of $M$. This follows from the definition, because Matlis duality gives an exact contravariant equivalence of the category of artinian $A$-modules with the category of finitely generated $A$-module. In this context, and for artinian modules only, the notions of coregular sequences (or cosequences) and codepth (or width) were defined by Matlis \cite{Matlis} in his original paper, and have been used by several authors since then. In this paper, however,  most of the modules we consider will not be artinian, and the old theory does not apply.
%This example is fine for motivation, but for the most part, we will be considering modules that will not be finitely generated or artinian. 
\end{example}

In the next theorem, inspired by the usual Koszul homology characterization of regular sequences, we use Koszul cohomology to characterize coregular sequences. 

\begin{theorem}\label{Kcoh} Let $A$ be a ring, $I$ an ideal, and $M$ a non-zero $A$-module with  $\Supp(M)\subset V(I)$.  A sequence  $x_1, \ldots, x_r$ of elements of $I$ is  coregular for $M$ if  and only if the Koszul cohomology groups $H^i(x_1, \ldots, x_r;M)$ are zero for all $i\ge 1$. 
\end{theorem}
\demo We proceed by induction on $r\ge 1$. If $r=1$, the Koszul complex is $M\stackrel{x_1}{\lto} M$. Notice that $H^0(x_1;M)=0:_Mx_1$ and  $H^1(x_1;M)=M/x_1M$. The latter is zero if and only if $x_1$ acts surjectively on $M$. All higher cohomologies are zero.  

If $r\ge2$ we use the short exact sequence of Koszul complexes
$$0\lto K^{\bullet}((x_1, \ldots, x_{r-1};M)[-1]\lto K^{\bullet}((x_1, \ldots, x_{r};M)\lto K^{\bullet}((x_1, \ldots, x_{r-1};M)\lto 0$$
The associated long exact sequence of cohomology is
$$ \stackrel{\delta_{i-1}}{\lto}H^{i-1}(x_1, \ldots, x_{r-1};M)\lto H^i(x_1, \ldots, x_r;M) \lto H^i(x_1, \ldots, x_{r-1};M)\stackrel{\delta_i}{\lto}H^i(x_1, \ldots, x_{r-1};M)$$
and the connecting homomorphisms $\delta_i$ are just multiplication by $x_r$ in the corresponding cohomology groups. 

Now suppose first that the sequence $x_1, \ldots, x_r$ is coregular for $M$. This implies that $x_1, \ldots, x_{r-1}$  is coregular for $M$ and that multiplication by $x_r$ acts surjectively on $$0:_M(x_1, \ldots, x_{r-1}) =H^{0}(x_1, \ldots, x_{r-1};M)\, .$$
Thus by induction $H^{i}(x_1, \ldots, x_{r-1};M)=0$ for all $i\ge 1$ and $\delta_0$ is surjective on $H^{0}(x_1, \ldots, x_{r-1};M)$. From the long exact sequence of cohomology we conclude that $H^{i}(x_1, \ldots, x_{r};M)=0$ for all $i\ge 1$. 

Now suppose, conversely, that $H^{i}(x_1, \ldots, x_{r};M)=0$ for all $i\ge 1$. Then by the long exact sequence of cohomology, the connecting homomorphisms $\delta_i$ are isomorphisms for $i\ge 1$ and surjective for $i=0$. However, the modules  $H^{i}(x_1, \ldots, x_{r-1};M)$ are all subquotients of sums of copies of $M$, hence have support in $V(I)$. Now by Lemma~\ref{ann}, every element of $M$ or any of its subquotients is annihilated by a power of $x_r$, since $x_r\in I$. Thus the only way multiplication by $x_r$ can be an isomorphism is that the modules $H^{i}(x_1, \ldots, x_{r-1};M)$ are zero. Hence, by induction, the sequence $x_1, \ldots, x_{r-1}$ is coregular for $M$. Furthermore, since $\delta_0$ is surjective on 
$H^{0}(x_1, \ldots, x_{r-1};M)=0:_M(x_1, \ldots, x_{r-1}) $, we see that $x_r$ acts surjectively on $0:_M(x_1, \ldots, x_{r-1}) $ and therefore the sequence $x_1, \ldots, x_r$ is coregular for $M$. 

\QED

\begin{corollary}\label{ORD} Let $A$ be a ring, $I$ an ideal, and $M$ a non-zero $A$-module with  $\Supp(M)\subset V(I)$. Any permutation of a coregular sequence  $\underline{x}\subset I$ for $M$ is a coregular sequence for $M$. 
\end{corollary}
\begin{proof}
It follows from Theorem~\ref{Kcoh} since the Koszul complex $K^{\bullet}(x_1, \ldots, x_n;M)$ is independent of the order of the $x_i$. 
\end{proof}

We must warn the reader that, unlike the case of regular sequences for finitely generated modules,  a partial coregular sequence cannot always be extended to a  coregular sequence  of length equal to the codepth (see, for instance, Example~\ref{TC}). Further, observe that a finitely generated module has always codepth zero. 

\medskip

\begin{proposition}\label{extcor} Let $A$ be a ring and $I$ an ideal.
%, and $M, M', M"$ non zero $A$-modules with support in $V(I)$.  
Consider a short exact sequence of $A$-modules with support contained in $V(I)$
$$0\lto M' \lto M\lto M''\lto 0\, .
$$
Let $\underline{x}=x_1, \ldots, x_r$  be a sequence of elements of $I$. 
\begin{enumerate}
\item
If $\underline{x}$ is a coregular sequence for $M'$ and $M''$ then it is a coregular sequence for  $M$. 
\item If $\underline{x}$ is a coregular sequence for $M$ then $\underline{x}$ is a coregular sequence for  $M''$ if and only if 
\[H^i(x_1, \ldots, x_{r}; M')=0\]
for all $i\ge 2$. 
\item If $x_1$ is coregular for $M$, it is also coregular for $M''$. 
\end{enumerate}
\end{proposition}
\begin{proof} 
These statements follow easily  from Theorem~\ref{Kcoh}  using the long exact sequence of Koszul cohomology of $x_1, \ldots, x_r$ associated to the short exact sequence of modules. 
\end{proof}

\begin{corollary}\label{c1.7} Let $A$ be a ring, $I$ an ideal, and $M$ a non-zero $A$-module with  $\Supp(M)\subset V(I)$. Let $h\in A$ with $M/hM\not=0$. If $x_1, x_2\in I $ is a coregular sequence for $M$, then it is also a coregular sequence for $M/hM$. 
\end{corollary}
\begin{proof} 
Let $\varphi $ be the map induced by multiplication by $h$ on $M$. Then we have the short exact sequences

$$0\lto \ker \varphi \lto M\lto\im \varphi\lto 0
$$
and 
$$0\lto \im \varphi \lto M\lto M/hM\lto 0\, .
$$
From the first exact sequence, since the functor $H^2(x_1, x_2; \bullet)$ is right exact, we conclude that $$H^2(x_1, x_2; \im \varphi)=0. $$ From Proposition~\ref{extcor}(b) and the second sequence, we conclude that $x_1, x_2$ is a coregular sequence for $M/hM$.  \end{proof}

\section{Hellus's theorem}\label{S2}

In this section we give a new version of Hellus's theorem, in terms of coregular sequences. It gives a criterion for an ideal to be generated up to radical by a regular sequence. In the case of the homogeneous coordinate ring of a projective space, it gives a criterion for a variety to be a set-theoretic complete intersection. 

\begin{theorem}\label{Hellus} Let $A$ be a ring, $I$ an ideal, and let $x_1, \ldots, x_r$ be an $A$-regular sequence contained in $I$. Then the following conditions are equivalent
\begin{itemize}
\item[(i)] $\sqrt{I} = \sqrt{(x_1, \ldots, x_r)} \, $
\item[(ii)]   
$x_1, \ldots,x_r$ form a coregular sequence for $H^r_I(A)$, and $H^i_I(A)=0$ for all $i>r$.

\end{itemize}

\end{theorem}

\begin{proof} First notice that since $x_1, \ldots, x_r$ form a regular sequence for $A$, we have $H^i_I(A)=0$ for all $i<r$. This is the usual local cohomology characterization of $I$-depth of $A$. 

$(i) \Longrightarrow (ii) $ 
The vanishing of the local cohomology 
%$i\not=r$ is well known: for $i <r$ it follows from the $I$-depth of $A$; 
for $i>r$  follows from the $\breve{\rm C}$ech computation of local cohomology. 

Since local cohomology depends only on the radical of the ideal, we may assume that $I=(x_1, \ldots, x_r)$. To show that $x_1, \ldots, x_r$ form a coregular sequence on $M=H^r_I(A)$, we use the Koszul complex $K^{\bullet}(x_1, \ldots, x_{r};M)$. Since $x_1, \ldots, x_r$ is  an $A$-regular sequence, the Koszul complex $K_{\bullet}(x_1, \ldots, x_{r})$ is a resolution of $A/I$ and therefore the cohomology groups of the complex $K^{\bullet}(x_1, \ldots, x_{r};M)$ are  the $\Ext^i(A/I, M)$. 
%Using the spectral sequence of Lemma~\ref{zero},  $\Ext^i(A/I, M)$  are isomorphic to $\Ext^{i+r}(A/I, A)$. 
In order to compute $\Ext^i(A/I, M)$, we will use an injective resolution of $M$ obtained from an injective resolution of $A$ by applying the functor $\Gamma_I$. Since the  $H^i_I(A)=0$ for all $i\not=r$, this complex is exact everywhere except for $i=r$ where its cohomology is $M$, and therefore it gives an injective resolution of $M$. This also shows that
$\Ext^i(A/I, M)=\Ext^{i+r}(A/I, A)$. But the latter modules are zero for all $i+r\not=r$ and hence $H^i(x_1, \ldots, x_r; M)=0$ for all $i>0$. Thus $x_1, \ldots, x_r$ form a coregular sequence for $M$ by Theorem~\ref{Kcoh}. 
\ms

\hspace{.6 cm}$(ii) \Longrightarrow (i) $ Since $x_1, \ldots, x_r$ is a regular sequence on $A$, we can write 
\[
0\lto A \overset{x_1}\longrightarrow A \lto A_1\lto 0
\]
where $A_1=A/(x_1)$. Running the long exact sequence of local cohomology with supports in $I$, we find, since $x_1$ acts surjectively on $M=H^r_I(A)$, that $M_1=H^{r-1}_I(A_1)$ is the only  local cohomology group of $A_1$ that is non zero. Since $M_1$ is the kernel of multiplication by $x_1$ on $M$, we find that $x_2, \ldots, x_r$ is a coregular sequence for $M_1$. 

Next we use the exact sequence 
\[
0\lto A_1 \overset{x_2}\longrightarrow A_1 \lto A_2\lto 0
\]
where $A_2=A/(x_1, x_2)$, and we find similarly that $M_2=H^{r-2}_I(A_2)$ is the only non-zero local cohomology group of $A_2$, and $x_3, \ldots, x_r$ is a coregular sequence for $M_2$. 

Proceeding inductively, we find that $A_r=A/(x_1, \ldots, x_r)$ has only one non-zero  local cohomology group $M_r=H^{0}_I(A_r)$. It follows by Lemma~\ref{zero} below that $A_r$ has support in $V(I)$. Since the annihilator of $A_r$ is $(x_1, \ldots, x_r)$, this shows that $I\subset \sqrt{(x_1, \ldots, x_r)}$. But $x_1, \ldots, x_r\in I$ by hypothesis, so $\sqrt{I}=\sqrt{(x_1, \ldots, x_r)}$. 

\end{proof}

\begin{lemma}\label{zero} Let $A$ be a ring, $I$ an ideal, and $N$ a finitely generated $A$-module with $IN\not=N$. If $H^i_I(N)=0$ for all $i>0$, then  $\Supp(N)\subset V(I)$ and $N=H^0_I(N)$. 
\end{lemma}
\begin{proof} We write the short exact sequence
\[0\lto H^0_I(N) \lto N \lto C\lto 0\]
where $C$ is another finitely generated $A$-module.  Our hypothesis implies that $H^i_I(C)=0$ for all $i$. But this is impossible unless $C=0$, because a non-zero finitely generated module has a well-defined $I$-depth, namely, the smallest $r$ such that $H^r_I(C)\not=0$ \cite[16.7]{M}.
\end{proof}

\begin{remark} In his paper, Hellus proves an equivalent version of Theorem~\ref{Hellus} for a Noetherian local ring $(A, \m)$. Instead of the notion of codepth, he uses the Matlis dual functor $D(\bullet)=\Hom_A(\bullet, E)$, where $E$ is an injective envelope of $k=A/\m$. Then his criterion is that $x_1, \ldots, x_r$ form a regular sequence for  $D(H^r_I(A))$. One can verify that this is equivalent to saying that $x_1, \ldots, x_r$ form a coregular sequence for $M$. 
\end{remark}

\begin{example}[\bf The twisted cubic curve]\label{TC} {\rm Let $\C$ be a twisted cubic curve in $\mathbb P^3$. One knows that $\C$ is a set-theoretic complete intersection, since it lies on a quadric cone $Q_0$, and on that cone, $2\C$ is a Cartier divisor cut out by a cubic surface $F$, so that $\C=Q_0\cap F$. Hence by Theorem~\ref{Hellus}, $M=H^2_I(A)$ has codepth 2, where $I$ is the ideal of $\C$ in the homogeneous coordinate ring $A$.

On the other hand, $\C$ lies on a non-singular quadric surface $Q$, and since $Q\setminus \C$ is affine (see \cite[\S V  1.10.1]{AG}), it follows that $H^2_I(A/qA)=0$, where $q$ is the equation defining $Q$. Therefore $q$ is a coregular element for $M$ (see  Proposition \ref{affine}). However, since the only complete intersections on $Q$ have bidegree $(a,a)$ for some $a>0$, whereas $\C$ has bidegree $(1,2)$ on $Q$,  no multiple of $\C$ can be a complete intersection on $Q$. 
% the curve $\C$ is not a set-theoretic complete intersection on $Q$. 
Therefore the coregular sequence $\{ q\}$ of length one cannot be extended to a coregular sequence of length 2, even though ${\rm codepth} \, M=2$.  }
\end{example}

\section{Quasi-cyclic modules}\label{S3}
In  the previous section, we saw that a curve $\C \subset \mathbb P^3$ is a set theoretic complete intersection if and only if the associated local cohomology module $M=H^2_I(A)$ has codepth $2$. This condition is difficult to verify in practice, so in this section we introduce another property of the module $M$, which may be easier to test. 

\begin{definition}
An $A$-module $M$ is {\it quasi-cyclic} if it is a countable increasing union of cyclic submodules. In the graded case we assume that the cyclic submodules are generated by homogenous elements.
\end{definition}

The facts listed in the following proposition are worth noting but easy to prove. We leave  the proofs to the reader. 

\begin{proposition}\label{PROP} Let $M$ be an $A$-module. \begin{itemize}
\item[(i)] $M$ is quasi-cyclic if and only if it is a countable direct limit of cyclic $A$-modules.
\item[(ii)] $M$ is quasi-cyclic if and only if it is countably generated and every finite subset of $M$  is contained in a cyclic submodule of $M$
\item[(iii)] Any quotient of a quasi-cyclic module is quasi-cyclic.
\end{itemize}
\end{proposition}

\medskip 

\begin{examples}~\label{EX} \begin{enumerate}
\item A finitely generated module is quasi-cyclic if and only if it is cyclic. 
\item If $f \in A$, the localization $A_f$ is quasi-cyclic.
\item If $A$ is a local Gorenstein ring and $\sqrt{J}= \sqrt{(f_1, \ldots, f_r)} \, $  for $f_1, \ldots, f_r$ a regular sequence, then the local cohomology module $H^r_J(R)$ is quasi-cyclic. Indeed,  it is the direct limit of ${\rm Ext}^r (R/ J_n, R)$, where $J_n=(f_1^n, \ldots, f_r^n)$, and  the module ${\rm Ext}^r (R/ J_n, R)$ is a canonical module $\omega_{A/J_n}$, which is  cyclic since $A/ J_n$ is Gorenstein.  
\end{enumerate}
\end{examples}

\begin{remark}\label{REM}{\rm  If a sequence  $x,y$ of elements of $I$ is coregular for $M$, then the sequence $x^{a}, y^{b}$  is coregular for $M$ for arbitrary positive integers $a,b$. For the proof notice that $0:_M x^a$ has a filtration by submodules $0:_M x^i$, for $1\le i \le a-1$,  whose quotients are isomorphic to $0:_M x$.}
\end{remark}

\begin{theorem}\label{cd>2} Let $A$ be a ring, $I$ an ideal, and $M$ a non-zero $A$-module with  $\Supp(M)\subset V(I)$.
 If $M$ is countably generated and has  codepth  at least $2$, then $M$ is quasi-cyclic. \end{theorem}
\begin{proof} By Proposition~\ref{PROP}(ii), it will be sufficient to show that any two elements $m,n \in M$ are contained in a cyclic submodule.  
Let $x,y\in I$ be a coregular sequence for $M$. By Corollary~\ref{ORD}, then $y,x$ is also a coregular sequence. By Lemma~\ref{ann} every element of $M$ is annihilated by some power of $x$ and some power of $y$. Let $N_i=0:_M x^i$ and let  $L_j=0:_M y^j$. Then $M=\cup_{i \ge 1}  N_i$ and $M=\cup_{j \ge 1}  L_j$. Thus we may assume that $m\in N_i$ and $n\in L_j$. By Remark~\ref{REM}, $x^i,y^j$ and $y^j,x^i$ are also coregular sequences for $M$. In particular, $y^j$ acts surjectively on $N_i$ and $x^i$ acts surjectively on $L_j$. In other words, there are elements $m'\in N_i$ and $n'\in L_j$ such that $y^jm'=m$ and $x^in'=n$. Now let $\alpha=m'+n'\in M$. Then 
\begin{eqnarray*} x^i \alpha=x^im'+x^in'=0+n=n\\
y^j \alpha=y^jm'+y^jn'=m+0=m\, .
\end{eqnarray*}
Hence $m$ and $n$ are contained in the cyclic submodule generated by $\alpha$. Thus $M$ is quasi-cyclic. 
\end{proof}

\begin{corollary} If a variety $V\subset  \mathbb P^n$ is a set-theoretic complete intersection of codimension $r$, then $H^r_I(A)$ is a graded quasi-cyclic $A$-module. 
\end{corollary}
\begin{proof} Write $M=H^r_I(A)$. If $r=1$ and $f$ is the defining equation of $V$, then $M=A_f/A$, which is quasi-cyclic by Example~\ref{EX}(2). If $r\ge 2$, then $M$ has codepth $r\ge 2$ by Theorem~\ref{Hellus}, and hence is quasi-cyclic by Theorem~\ref{cd>2}. 
\end{proof}

\begin{example}\label{skew} Let $\C$ be the disjoint union of two lines $L_1$ and $L_2$ in $\mathbb P^3$. Then $H^2_I(A)$, where $I$  is the defining  ideal of $\C$, is not quasi-cyclic. In particular, this gives a new proof of the well-known result that $\C$ is not a set-theoretic complete intersection. 
\end{example}
\begin{proof} By direct computation. If we define the lines by $x=y=0$ and $z=w=0$, then  the Mayer-Vietoris sequence for local cohomology
\[ 0=H^2_{(x, y, z, w)}(A) \lto H^2_{(x,y)}(A) \oplus H^2_{(z,w)} (A)  \lto H^2_I(A) \lto H^3_{(x, y, z, w)}(A)=0
\]
shows that  $ H^2_I(A)\cong H^2_{(x,y)}(A) \oplus H^2_{(z,w)} (A) $. Let
$M_1=H^2_{(x,y}(A)$ and $M_2=H^2_{(z,w)} (A)$ and notice that they are Macaulay inverse systems of $x,y$ (of $z,w$ respectively) on $A$. Thus they are generated as $k[z, w]$-module (as $k[x, y]$-module, respectively) by $\{x^iy^j \ | \ i<0 \ \mbox{and} \ j<0\}$ (respectively, by $\{z^iw^j \ | \ i<0 \ \mbox{and} \ j<0\}$). 

An arbitrary element of $M$ can be written as $$\alpha=\frac{a}{x^iy^j}+\frac{b}{z^{\ell}w^m}$$
for positive integers $i,j, \ell, m$ and $a,b\in A$. 

We now prove that the the socle elements of $M$  
$$\frac{1}{xy} \qquad \mbox{\rm and} \qquad \frac{1}{zw}$$
are not contained in any cyclic submodule of $M$, that is they cannot be contained in the submodule generated by any $\alpha$. 

Indeed, suppose  $f\alpha=\frac{1}{xy} $ and $g\alpha= \frac{1}{zw}$ for some $f, g\in A$. The equality $f\alpha=\frac{1}{xy} $ forces $b=0$, since otherwise $f\in (z,w)$ in which case the first term would have $z$ and $w$. But if $b=0$ the equality $g\alpha= \frac{1}{zw}$ cannot be satisfied. 
\end{proof}

\begin{remark}[Segre embedding] If $V$ is the Segre embedding of $\mathbb P^1 \times \mathbb P^2$ into $\mathbb P^5$, then Claudiu Raicu has shown us an argument, using representation theory in characteristic zero, that $H^2_I(A)$ is not quasi-cyclic. The point is that $H^2_I(A)$, with its grading, is a direct sum of irreducible representations and Pieri's rules shows that the `bottom' piece cannot be reached by multiplication from higher pieces. This gives another proof in characteristic zero that $V$ is not a set-theoretic complete intersection, which was already known, since $H^3_I(A)$ is not zero. 
\end{remark}

\section{Curves on surfaces}\label{S4} 

In this section we consider the following Setting~\ref{ass} to relate properties of $X$ and $X\setminus \C$ to the behavior of multiplication by $h$ on $M$ and the module $M/hM$. In particular we examine when $X\setminus \C$ is affine or a modification of an affine. 

\begin{setting}\label{ass}{ \rm Let  $\C$ be a connected curve lying on a surface $X$ in $\mathbb P^3$. Let $A$ be the homogeneous coordinate ring of $\mathbb P^3$, let $I$ be the ideal of $\C$, and let $h \in I$ be the element defining the surface $X$. Write $M$ for the module $H^2_I(A)$. }
\end{setting}

\begin{proposition}\label{affine} In Setting~\ref{ass}, $h$ is a coregular element for $M$ if and only if $X\setminus \C$ is affine. 
\end{proposition}
\begin{proof} Let $B=A/(h)$  be the coordinate ring of $X$ and let $J=IB$ be the ideal of $\C$ in $X$. Note that $M/hM\cong H^2_I(A/(h))$ since $H^3_I(A)=0$ according to \cite[7.5]{CDAV}. Hence $h$ is coregular for $M$ if and only if $H^2_I(A/(h))=H^2_I(B)=0$. We make use of the well-known result (see \cite[Section 2, page 412]{CDAV}) that 
\[H^2_J(B)\cong \oplus_{n \in \mathbb Z}H^1(X \setminus \mathcal C, \mathcal O_X(n))\, .\]
The latter is zero if and only if $H^1(X 
\setminus \mathcal C, \mathcal O_X(n))=0$ for all $n \in \mathbb Z$. Since $H^i_J(B)=0$ for all $i \ge 3$, it follows that also $H^i(X \setminus \mathcal C, \mathcal O_X(n))=0$  for all $i\ge 2$, and so $H^1(X \setminus \mathcal C, \mathcal F))=0$ for all coherent sheaves on $X \setminus \mathcal C$. By Serre's criterion, this implies that $X \setminus \mathcal C$ is affine. Conversely, if $X \setminus \mathcal C$ is affine all higher cohomology of coherent sheaves on $X \setminus \mathcal C$ vanish, so $h$ is coregular. 
\end{proof}

\begin{example}\label{e5.3}
\begin{enumerate}[(a)]
\item Let $\C$ be a curve of bidegree $(a,b)$ on a nonsingular quadric surface $Q$ with $a,b>0$. Then $\C$ is connected and the  defining equation $q$ of  $Q$ is coregular for $M$. Indeed, $\C$ is ample on $Q$ so $Q\setminus \C$ is affine. 
\item Let $\C$ be a curve on a nonsingular cubic surface $X$ with divisor class $(a;b_1, \ldots, b_6)$ in the usual notation, with $a\ge b_1+b_2+b_3$ and $b_1\ge b_2\ge \ldots \ge b_6>0$. Then $\C$ is connected and is ample on $X$ (see \cite[V, 4.12]{AG}). Hence $X\setminus \C$ is affine, and the equation defining $X$ is coregular for $M$. 
\item Let $\C$ be {\bf the rational quartic curve} given by the parametric representation $$(x,y,z,w)=(t^4,t^3u,tu^3,u^4)\, .$$ This curve $\C$ lies on the nonsingular quadric surface defined by $q=xw-yz$ and it has bidegree $(1,3)$ hence by (a), $q$ is a coregular element for $M$. The curve $\C$ also lies on the cone $X$ over a cuspidal plane cubic curve defined by $g=y^3-x^2z$.  We will show that $X\setminus \C$ is affine, so that $g$ is another coregular element for $M$ (the same applies to the element $h=z^3-yw^2$). However, $g,q$ is not a coregular sequence, by Hellus's theorem, since the ideal $(g,q)$ defines a curve of degree 6 consisting of $\C$ and a double line. 

To show that $X\setminus \C$ is affine, we consider the normalization $\tilde{X}$ of $X$, which is the cone over a twisted cubic curve in $\mathbb P^3$ \cite[6.9]{HP}. The inverse image of $\C$ in $\tilde{X}$ is a curve $\tilde{\C}$, isomorphic to $\C$, which meets every ruling of the cone $\tilde{X}$ in one point. Therefore $3\tilde{\C}$ is a Cartier divisor on $\tilde{X}$, and since ${\rm Pic} \, \tilde{X}\cong \mathbb Z$, this Cartier divisor is ample hence $\tilde{X}\setminus \tilde{\C}$ is affine. But $\tilde{X}\setminus \tilde{\C}$ is the normalization of $X\setminus \C$, thus $X\setminus \C$ is affine as well. 
\end{enumerate} 
\end{example}

\begin{definition}\label{modification} A scheme $V$ of finite type over $k$ is a {\it modification of an affine} if there exists a proper surjective map $\pi: V  \lto V_0$ with  $V_0$ affine, $\pi_* \mathcal O_V=\mathcal O_{V_0}$, and such that $\pi$ has only finitely many {\it fundamental points}, that is points $P\in V_0$ for which ${\rm dim}\, \pi^{-1}(P)\ge 1$. 
\end{definition}

\begin{proposition}\label{4.4} In Setting~\ref{ass} the following conditions are equivalent:
\begin{enumerate} 
\item $X\setminus \C$ is a modification of an affine scheme;
\item ${\rm dim}\, H^i(X\setminus \C, \mathcal F) < \infty\ $ for every coherent sheaf $\mathcal F$ on $X\setminus \C$ and every $i>0$;
\item each graded component of $M/hM=H^2_I(A/hA)$ is finite dimensional over $k$. 
\end{enumerate}
\end{proposition}
\demo
The equivalence of (1) and (2) follows from \cite[Theorem 1 and Corollary 3]{GH}. 

We now show (2) implies (3). As in the proof of Proposition~\ref{affine}, we have
$$M/hM\cong \oplus_{n \in \mathbb Z}H^1(X \setminus \mathcal C, \mathcal O_X(n))$$ where $n$ indicates the grading. Thus by (2), each graded piece is finite-dimensional. 

Finally we show that (3) implies (2). Again, as in the proof of Proposition~\ref{affine}, $H^i(X \setminus \mathcal C, \mathcal O_X(n))=0$ for all $n\in \mathbb Z$ and all $i\ge2$. Thus by the usual d\'evissage, $H^i(X \setminus \mathcal C, \mathcal F)=0$ for $i\ge 2$ and finite dimensional for $i=1$, for all coherent sheaves $\mathcal F$ on $X$. 
\QED

\begin{corollary} In Setting~\ref{ass} if $\C$ satisfies the equivalent conditions of Proposition~\ref{4.4}, then the degree $n$ component of $M/hM$ is zero for $n\gg 0$. 
\end{corollary}
\demo This follows from \cite[Corollary 4]{GH}. 
\QED

\begin{example} \label{e5.7}
\begin{enumerate}[(a)]
\item If $\C$ is a curve on the nonsingular cubic surface $X$ having divisor class $(a;b_1, \ldots, b_6)$ in the usual notation, with $a\ge b_1+b_2+b_3$ and $b_1\ge b_2\ge\ldots \ge b_6\ge 0$ and $b_3>0$, then $X\setminus \mathcal C$ is a modification of an affine scheme. Indeed, in the trivial case $b_6>0$, we have seen already that $X\setminus \C$ is affine (see Example~\ref{e5.3}(b)), so there is nothing to prove. If $r$ is the largest index for which $b_r>0$, with $r=3,4,5$, then $\C$ is on a surface $X_0$ obtained by blowing up $r$ points in $\mathbb P^2$, and $\C$ will be ample there, so that $X_0\setminus \C$ is affine. Then the projection $\pi: X\setminus \C \lto X_0 \setminus \C$ makes $X\setminus \C$ a modification of an affine.  
\item If $\C$ is the rational quartic curve of Example~\ref{e5.3}(c), then $\C$ lies on a ruled cubic surface $X$ defined by $p=xz^2-y^2w$ (see \cite[\S 6]{GD} and \cite[6.3 and 7.12]{HP}). Since $\C$ meets the double line of $X$ only at the pinch points, its inverse image in the normalization $S$ of $X$  is a curve $\C'$ corresponding to a conic $\C_0$ in $\mathbb P^2$ that does not meet the point $P$ of $\mathbb P^2$ that was blown up. Thus $S\setminus \C'$ is a modification of the affine scheme $\mathbb P^2 \setminus \C_0$. Now it follows that $X\setminus \C$ is a modification of an affine since $\pi: S\setminus \C' \lto X\setminus \C$ is a finite morphism. Indeed, by applying the criterion of Proposition~\ref{4.4}(2) and  the Leray spectral sequence of cohomology for a finite surjective morphism $\pi: V'\lto V$, we obtain that if $V'$ is a modification of an affine if and only if $V$ is. 
\item More generally, if $\C$ is a curve on a nonsingular surface $X$ with $\C^2>0$, then $X\setminus \C$ is a modification of an affine. Inded, this follows from \cite[\S III 4.1]{ASAV} which implies that $H^i(X\setminus \C, \mathcal F) $ is finite-dimensional for every coherent sheaf $\mathcal F$ on $X\setminus \C$ and every $i>0$.
\end{enumerate}
\end{example}

\begin{theorem} In Setting~\ref{ass} if $X\setminus \C$ is a modification of an affine, then the $I$-codepth of $M/hM$ is at least two. 
\end{theorem}
\demo Since $X\setminus \C$ is a modification of an affine, we have a proper map $\pi: X\setminus \C \lto V$ with $V$ affine. Then there are finitely many fundamental points $P_1, \ldots, P_s \in V$. If we let $E_i=\pi^{-1}(P_i)$, then $\pi: X\setminus \C \lto V$ is an isomorphism outside $E_i$ and $P_i$. Using the Leray spectral sequence for $\pi$ and the fact that $V$ is affine, we find
$$ (M/hM)_n=H^1(X\setminus \C, \mathcal O(n))=H^0(V, R^1\pi_* \mathcal O(n))\, .
$$
The sheaves $R^1\pi_* \mathcal O(n)$ are coherent and are supported at the points $P_i$. Therefore, 
$$H^0(V, R^1\pi_* \mathcal O(n))=R^1\pi_* \mathcal O(n)\, .$$

In addition  the modules $R^1\pi_* \mathcal O(n)$ are of finite length and are equal to their own completions over the local rings at $P_i$. 
Write $E=\cup E_i$. Now the theorem on formal functions \cite[III, 11.1]{AG} shows that
$$ R^1\pi_* \mathcal O(n)=\varprojlim H^1(E_{\nu}, \mathcal O_{E_{\nu}}(n))
$$
where $E_{\nu}$ is the closed subcheme of $X\setminus \C$ defined by $\mathcal I^{\nu}_E$, where $\mathcal I_E=\pi^*(\sum \m_{P_i})$ defines the inverse image scheme of the $P_i$. The maps in the inverse system come from the short exact sequences
$$0\lto \mathcal I_E^{\nu}/\mathcal I_E^{\nu+1} \lto \mathcal O_{E_{\nu+1}}\lto  \mathcal O_{E_{\nu}} \lto 0\, .
$$
Since the $E_{\nu} $ are curves, the $H^1$ functor is right exact, so the maps of the inverse system are all surjective. Furthermore, since $(M/hM)_n$ is finite-dimensional for each $n$, it follows that the maps in the inverse system, for each $n$, are eventually constant. 

Now we look for coregular elements. Since elements of the ideal $I$ cut out the curve $\mathcal C$ we can find elements $f,g\in I$ whose zero-sets meet the curve $E$ in finitely many points all distinct from each other. It may happen that the schemes $E_{\nu}$ for various $\nu$, have embedded points. However, according to a theorem of Brodmann \cite[1]{B}, the union for all 
$\nu$ of the sets of embedded points of $E_{\nu}$ is a finite set. Therefore, we may assume that $f,g$ meet each $E_{\nu}$ in distinct points, none of which is embedded for $E_{\nu}$.

Since $f$ does not meet $E_{\nu}$ at any embedded points, we have exact sequences
$$0\lto \mathcal O_{E_{\nu}}(n) \stackrel{f}\lto \mathcal  O_{E_{\nu}}(n+d) \lto \mathcal  O_{Z_{\nu}}(n+d) \lto 0$$
where $Z_{\nu}$ is the scheme of zeros of $f$ in $E_{\nu}$. Since $Z_{\nu}$ is a zero dimensional scheme, $H^1(E_{\nu}, \mathcal  O_{Z_{\nu}}(n+d))=0$. Hence multiplication by $f$ is surjective on $\oplus_n H^1(E_{\nu},  \mathcal O_{E_{\nu}}(n)))=H^1(E_{\nu}, \oplus_n \mathcal O_{E_{\nu}}(n)))$ for each $\nu$. Since the maps of the inverse limit above are eventually constant in each degree, it follows that multiplication by $f$ is surjective on $M/hM$. Thus $f$ is a coregular element for $M/hM$. 

The kernel of multiplication by $f$ on $H^1(E_{\nu}, \mathcal O_{E_{\nu}}(n))$ is a quotient of  $H^0(E_{\nu}, \mathcal  O_{Z_{\nu}}(n+d))$. Since $g$ does not vanish at the points of $Z_{\nu}$, it follows that $g$ acts isomorphically on $ \oplus_n H^0(E_{\nu}, \mathcal O_{E_{\nu}}(n)))=H^0(E_{\nu}, \oplus_n \mathcal O_{E_{\nu}}(n)))$. Thus multiplication by $g$ is surjective on  the kernel $0:_{H^1(E_{\nu},  \oplus_n \mathcal O_{E_{\nu}}(n))} (f)$ for each $\nu$. Again, since the maps in the inverse system are eventually isomorphisms, it follows that multiplication by $g$ is surjective on $0:_{M/hM}(f)$, so that $(f,g)$ form a coregular sequence of length two for $M/hM$. 
\QED

\bs

\section{Intersection of three surfaces}\label{S5}

The first example was in Perron's paper \cite{P}, where he showed that Vahlen's curve, a rational quintic curve with a single quadrisecant, is an intersection of three surfaces in $\mathbb P^3$. Then Kneser \cite{Kn} showed that any curve in $\mathbb P^3$ is the intersection of three surfaces. This was generalized by Eisenbud and Evans \cite{EE}, and independently Storch \cite{S} (in the affine case only), to show that any variety in affine or projective $n$-space is an intersection of $n$ hypersurfaces.

In this section 
we interpret the condition for three polynomials   $f, g, h$ to cut out a subvariety of codimension 2 in $\mathbb P^n$ set-theoretically, in terms of the Koszul cohomology of certain local cohomology modules. We start with an auxiliary result, which will be used for the homogeneous coordinate ring of a hypersurface.

\begin{proposition}\label{6.1} Let $B$ be a ring, let $J$ be an ideal, let $f,g\in J,$ and assume that $f$ is not a zero-divisor in $B$. The following condition are equivalent  
\begin{enumerate}
\item $\sqrt{J}=\sqrt{(f,g)}$
\item  Let $M_i=H^i_J(B)$ for $i\ge1$
\begin{enumerate}
\item $M_i=0$ for $i>2$
\item $f,g$ is a coregular sequence for $M_2$
\item The natural map  (defined in the proof)
$$\delta: H^0(f,g;M_2)\lto H^2(f,g; M_1)$$
is surjective.
\end{enumerate}
\end{enumerate}
\end{proposition}
\begin{proof} If $f,g$ is a regular sequence in $B$, then $M_1$ is automatically zero and this statement follows from Theorem~\ref{Hellus}.

We first show that (1) implies (2). Part (a) follows from the computation of local cohomology using the $\check{\rm C}$ech complex.

For part (b) consider the short exact sequence
$$0\lto B \stackrel{f}\lto B \lto B/fB\lto 0\, .$$
Applying local cohomology we obtain the long exact sequence
$$\ldots \lto M_1 \stackrel{f}\lto M_1 \lto H^1_J(B/fB) \lto M_2 \stackrel{f}\lto M_2 \lto H^2_J(B/fB)\lto 0$$

We may assume that $J$ is defined by $f,g$. Then in the ring $B/fB$ the ideal $J$ is defined by $g$, so $H^2_J(B/fB)=0$. We conclude that multiplication by $f$ is surjective on $M_2$. 

Write $K={\rm ker}(M_2 \stackrel{f}\lto M_2)$ and $Q={\rm coker} (M_1 \stackrel{f}\lto M_1 )$. Thus we obtain a short exact sequence 
$$0\lto Q \lto H^1_J(B/fB) \lto K\lto 0\, .$$
Multiplication by $g$ gives a commutative diagram

$$\begin{tikzcd} 0 \arrow{r} &Q \arrow{r} \arrow{d}{g} & H^1_J(B/fB)  \arrow[two heads]{d}{g}\arrow{r} &K\arrow{d}{g} \arrow{r} &0 \\
0 \arrow{r} &Q \arrow{r}  & H^1_J(B/fB) \arrow{r} & K \arrow{r} &0 
 \end{tikzcd} \, .$$

The middle map is surjective since $J(B/fB)$ is generated by $g$, so using the Snake Lemma, we find that the coboundary map
$$\delta: {\rm ker}(K \stackrel{g}\lto K)\lto {\rm coker} (Q \stackrel{g}\lto Q )$$
is surjective, and the last map in the diagram $K \stackrel{g}\lto K$  is surjective as well. This last assertion shows by definition that $f,g$ is a coregular sequence for $M_2$.

To prove (c) we need only to observe that  ${\rm ker}(K \stackrel{g}\lto K)=H^0(f,g;M_2)$ and ${\rm coker} (Q \stackrel{g}\lto Q )=H^2(f,g; M_1)$. 

We now show that (2) implies (1). Using the hypotheses (a), (b), (c), and running the argument backwards, we find 
$$H^2_J(B/fB)=0 \quad \mbox{and} \quad H^1_J(B/fB)\stackrel{g}\lto H^1_J(B/fB) \quad \mbox{is surjective}\, .$$
Since $H^i_J(B/fB)=0$ for $i\ge 2$, a standard d\'evissage shows that $H^i_J(N)=0$ for any finitely generated $B/fB$-module. Thus the functor $H^1_J(-)$ is right exact  for finitely generated $B/fB$-modules. Therefore $H^i_J(B/(f,g)B)=0$ for all $i>0$. The latter implies that ${\rm Supp}(B/(f,g)B)\subset V(J)$ according to Lemma~\ref{zero}, hence (1) holds. 
\end{proof}

\begin{theorem} Let $A$ be a ring, $I$ an ideal, and let $f,g,h\in I$. Assume that $(f,h)$ is a regular sequence in $A$. Assume also that $H^i_I(A)=0$ for $i>2$, and let $M=H^2_I(A)$. Then the following are equivalent
\begin{enumerate}
\item $\sqrt{I}=\sqrt{(f,g,h)}$
\item $H^i(f,g,h;M)=0$ for $i\ge 2$.
\end{enumerate}
\end{theorem}
\begin{proof}
Let $B=A/hA$, let $J=IB$, and notice that $f$ is regular on $B$. We want to apply Proposition~\ref{6.1} to $B$, $J$, and $f,g \in J$. Notice that part (1) of Proposition~\ref{6.1} is clearly equivalent to   $\sqrt{I}=\sqrt{f,g,h}$. So it is enough to show that $H^i(f,g,h;M)=0$ for $i\ge 2$ is equivalent to assertion (2) of Proposition~\ref{6.1}.

We use the notation of Proposition~\ref{6.1}, in particular we let $M_i=H^i_J(B)$ for $i\ge1$. Since $h$ is a regular element in $A$ and $H^i_I(A)=0$ for $i\not=2$, we have the exact sequence
$$ 0\lto M_1\lto M\stackrel{h}\lto M\lto M_2\lto0\,,
$$
and note that all $M_i=0$ for $i>2$. We need to show that $H^i(f,g,h;M)=0$ for $i\ge 2$ is equivalent to (b) and (c) of Proposition~\ref{6.1}. 

Consider the map of Koszul complexes $K^{\bullet}(f,g;M)\stackrel{h}\lto K^{\bullet}(f,g;M)$. The cohomology of the total complex is simply $H^i(f,g,h;M)$. The spectral sequence of the double complex degenerates to give a long exact sequence
$$0\lto H^1(f,g; M_1)\lto H^1(f,g,h; M)\lto H^0(f,g; M_2)\stackrel{\delta}\lto $$
$$\hspace{1cm}H^2(f,g; M_1)\lto H^2(f,g,h; M)\lto H^1(f,g; M_2)\lto 0$$
and the isomorphism
$$H^3(f,g,h; M)\cong H^2(f,g; M_2)\, .$$

The vanishing of $H^2(f,g,h; M)$ and $H^3(f,g,h; M)$ is equivalent to the vanishing of $H^1(f,g; M_2)$ and $H^2(f,g; M_2)$, plus the surjectivity of $\delta$. Hence $H^i(f,g,h;M)=0$ for $i\ge 2$ is equivalent to (b) and (c) of Proposition~\ref{6.1} by Theorem~\ref{Kcoh}. 
\end{proof}

\begin{corollary} If $\mathcal C$ is any curve in $\mathbb P^3$, then there exists a surface $X$ containing $\mathcal C$, defined by $h=0$, such that $M/hM$ has $I$-codepth 2, where $M=H^2_I(A)$. 

\end{corollary}
\begin{proof} By Kneser's Theorem, there exist $f,g,h$ with $\sqrt{I}=\sqrt{(f,g,h)}$. We may assume that two of these, say $f,h$, form a regular sequence for $A$. Then if $X$ is defined by $h$, and $B=A/hA$, applying Proposition~\ref{6.1}, we see that $f,g$ form a coregular sequence for $M_2=H^2_J(B)=M/hM$. 
\end{proof}

\bs

\section{Questions}\label{S6}

The first question is a well known, long-standing open problem, which was the motivation for this paper. The remaining questions are consequences. That is to say, a yes answer to the first question would imply a yes answer to all the remaining questions. Conversely a no answer to any of the remaining questions would imply a no answer to the first question. In all these questions $\mathcal C$ is an irreducible curve in $\mathbb P^3$ and $I$ is its homogeneous ideal in the homogeneous coordinate ring $A$ of $\mathbb P^3$.

\begin{question}\label{Q1}{\rm Is every irreducible curve in $\mathbb P^3$ a set-theoretic complete intersection? }
\end{question}

As far as we know, this question was first stated explicitly by Perron.

\begin{question}\label{Q2}{\rm If $\mathcal C$ is an irreducible curve in $\mathbb P^3$, does the module $H^2_I(A)$ have codepth $2$? }
\end{question}

By Theorem \ref{Hellus},  Question ~\ref{Q2}  is equivalent to Question \ref{Q1}. 
\begin{question}\label{Q3}{\rm If $\mathcal C$ is an irreducible curve in $\mathbb P^3$, is the module $H^2_I(A)$ quasi cyclic?} 
\end{question}

Question \ref{Q3} follows from Question \ref{Q2} by Theorem \ref{cd>2}. 

\begin{question}\label{Q4}{\rm If $\mathcal C$ is an irreducible curve in $\mathbb P^3$, does the module $H^2_I(A)$ have codepth $\ge 1$? } 
\end{question}

Clearly Question \ref{Q4} is a trivial consequence of Question~\ref{Q2}

\begin{question}\label{Q5}{\rm If $\mathcal C$ is an irreducible curve in $\mathbb P^3$, does there exist a surface $X$ containing $\C$ with $X \setminus \C$ affine?} 
\end{question}

Question \ref{Q5} is equivalent to Question~\ref{Q4} by Proposition~\ref{affine}.

\begin{question}\label{Q6}{\rm If $\mathcal C$ is an irreducible curve in $\mathbb P^3$, does there exist a surface $X$ containing $\C$ with $X \setminus \C$ a modification of an affine?} 
\end{question}

Question \ref{Q6} is a trivial consequence of Question~\ref{Q5}.

\begin{question}\label{Q7}{\rm If $X$ is an irreducible surface containing  an irreducible curve $\mathcal C$ in $\mathbb P^3$, do there exist two more surfaces $Y$ and $Z$ containing $\C$ such that $\C=X\cap Y\cap Z$?} 
\end{question}

Question \ref{Q7} would follow from Question~\ref{Q1}, taking $Y$ and $Z$ to be the surfaces defining $\C$. 

\begin{question}\label{Q8}{\rm If $\mathcal C$ is an irreducible curve in $\mathbb P^3$ and $X$ is a surface defined by a non-coregular element $h\in I$, does the module $H^2_I(A/hA)$ have codepth $2$? } 
\end{question}

Question \ref{Q8} would follow from  Question~\ref{Q2} by Corollary~\ref{ORD}. It would also follow from Question~\ref{Q7} by Proposition~\ref{6.1}. 

\begin{question}\label{Q9}{\rm With $\C, X, h$ as in Question~\ref{Q8}, is the module $H^2_I(A/hA)$ quasi-cyclic? } 
\end{question}

Question \ref{Q9} follows from Question~\ref{Q8} by Proposition~\ref{cd>2}, or from Question~\ref{Q3} by Proposition~\ref{PROP}(iii).

\begin{example} 
\begin{enumerate}
\item  The rational quartic curve (see Example~\ref{e5.3}(c)) is the smallest degree example of a curve for which it is not known whether it is a set-theoretic complete intersection in characteristic zero. In any characteristic $p>0$ it is a set-theoretic complete intersection \cite{CI}. For this curve, the answers to Questions~\ref{Q4} and \ref{Q5} are yes. 
\item For Vahlen's quintic, which is a rational quintic curve $\C$ with a single quadrisecant, we do not know the answer to Question~\ref{Q5}, but the answer to Question~\ref{Q6} is yes, because $\C$ can be found on a nonsingular cubic surface $X$ with divisor class $(2;1, 10^5)$, using Example~\ref{e5.7}(c). 
\item We will show that the answer to Question~\ref{Q7} is also yes for Vahlen's quintic. Let $X$ be a nonsingular cubic surface and let $\C$ be the curve with divisor class $(2;1, 10^5)$. We take $Y$ to be another cubic surface containing $\C$ so that $X\cap Y=\C\cup G\cup T$ where $X\cap Y$ has divisor class $3H=(9;3^6)$ and $G$ is the line $(2; 0,1^5)$, which is the quadrisecant, and $T=(5; 2^6)$ is a twisted cubic curve. We can take $T$ to be irreducible and nonsingular, because for any $T$ in that divisor class, $C\cup G\cup T\sim 3H$, and $X$ being projectively normal, this is the intersection with another cubic surface $Y$. 

We cannot take $Z$ to be another cubic surface, because every cubic surface containing $\C$ also contains its quadrisecant $G$. So we look for a quartic surface $Z$. Then $X\cap Z=C\cup D$ where $D=(10; 3,4^5)$. This is a curve of degree 7 and genus 3. Any curve in this linear system arises as $X\cap Z \setminus \C$. 

To show that $\C=X\cap Y\cap Z$, we must show that the points of  $D\cap (G\cup T)$ are contained in $\C$. Now $D\cdot G=0$, and a general $D$ is irreducible so we may assume $D\cap G=\emptyset$. Observe that $D\cdot T=4$. We need to show that $D$ can be chosen so that the four points of $D\cap T$ are among the 8 intersections of $T$ with $\C$. Since $T$ is a rational curve, for this it will be sufficient to show that the linear system $|D|$ on $X$ maps surjectively to the linear system $|D\cdot T|$ on $T$; that is that the map 
$$\pi: H^0(\mathcal O_X(D))\lto H^0(\mathcal O_T(D\cdot T))$$
is surjective. The cokernel of $\pi$  is $H^1(\mathcal O_X(D\setminus T))$, as can be seen applying cohomology to the short exact sequence
$$ 0 \lto \mathcal O_X(D\setminus T) \lto \mathcal O_X(D) \lto \mathcal O_T(D \cdot T)\lto 0\, .
$$ Now $D\setminus T$ has divisor class $(5;1,2^5)$ and therefore can  be represented by an elliptic curve $E$ of degree 4. From the exact sequence
$$\ldots \lto H^1(\mathcal O_X) \lto H^1(\mathcal O_X(E)) \lto H^1(\mathcal O_E(E)) \lto \ldots
$$
we deduce that $H^1(\mathcal O_X(E))=0$,  since $H^1(\mathcal O_X)=0$ and $E^2=4$ which gives $H^1(\mathcal O_E(E))=0$. Thus the map $\pi$ is surjective. 

Hence we conclude that $\C=X\cap Y\cap Z$, the intersection of two cubics and one quartic surface. 
\end{enumerate}
\end{example}

\bs
\noindent{\bf Acknowledgments.} 
Part of this paper was written at the 
 Centre International de Rencontres Math\'emati\-ques (CIRM) in Luminy, France, while the authors participated in a Recherches en Bin\^{o}me.  We are very appreciative of the hospitality offered by the Soci\'et\'e Math\'ematique de France.

 \bs

\end{document}